\documentclass{amsart}
\usepackage{amssymb, amsmath, amsthm}
\usepackage{hyperref}

\title[Theta constants for 3-sheeted covers of the sphere]{Theta constants identities for Jacobians of cyclic 3-sheeted covers of the sphere and representations of the symmetric group}
\author[Kopeliovich]{by Yaacov Kopeliovich}
\dedicatory{To my friend Elizabeth Drake}
\date{ 04 April 2007}
\address{5736 Las Virgenes Rd. Calabasas CA 91302}
\newtheorem{thm}{Theorem}[section]
\newtheorem{cor}[thm]{Corollary}
\newtheorem{lem}[thm]{Lemma}

\theoremstyle{definition}
\newtheorem{defn}[thm]{Definition}
\theoremstyle{Assumption}

\theoremstyle{remark}

\begin{document}
\begin{abstract}
We find identities between theta constants with rational
characteristics evaluated at period matrix of  $R,$  a cyclic 3
sheeted cover of the sphere with $3k$ branch points
$\lambda_1...\lambda_{3k}.$ These identities follow from Thomae
formula \cite{BR}. This formula expresses powers of theta constants
as polynomials in $\lambda_1...\lambda_{3k}.$ We apply the
representation of the symmetric group to find relations between the
polynomials and hence between the associated theta constants.
\end{abstract}
\maketitle

\section{Introduction}

Let $R$ be a Riemann surface with the equation: $$y^3 =
\prod_{i=1}^{i=3m}(z-\lambda_i)(*).$$ We find relations that are
satisfied by theta constants with rational characteristics evaluated
at $\tau_R,$ the period matrix of $R.$ Special type identities for
period matrices are known in the case of a general Riemann surface (
Schottky-Jung identities). For hyperelliptic curves there are
vanishing theta constants of even characteristics that characterize
the associated period matrix. According to Mumford, \cite{Mu}
special relations of non vanishing of theta constants  evaluated at
period matrices of hyperelliptic curves were obtained by Frobenius.

The original Schottky problem seeks special relations among theta
constants that characterize the entire moduli space of algebraic
curves of genus $g.$ In this note we seek special relations that are
satisfied by  $n$-sheeted cyclic covers of the sphere. When $n=2$
cyclic covers are just hyperelliptic curves. The next case is $n=3$
and we find relations between theta constants with rational
characteristics evaluated at $\tau_R$ the period matrices of such
curves .

These identities are a result of  Thomae formula for cyclic $n$
sheeted covers of the sphere. This formula expresses powers of such
theta constants evaluated at the period matrix $\tau_R$ through
polynomial expression of $\lambda_i$. A relation between these
polynomials produces a relation between associated theta constants.
Applying the  representation theory of the symmetric group, $S_{3m}$
we produce  a basis for the vector space spanned by the polynomials
and as a result relations between the associated theta constants.

\smallskip
For the simplest case of $6$ branch points our results overlap with
results of Matsumoto \cite{Ma}. In his paper Matsumoto finds the
explicit action of $S_6$ on theta constants evaluated at $\tau_R$
and expresses branch points $\lambda_i$ as rational functions of
theta constants. As a result he writes identities between cubic
powers of these constants which essentially coincide with the
identities obtained by us in the last section of our note. Using the
representation theory of $S_6$ we see that the space generated by
theta constants is 5 dimensional. This seems to be a new result even
in this case. We note that the Algebraic dimension of this
particular family of curves is 3.

This work was partially done during a visit to the TAMU math
department and the author thanks the department for the invitation
and kind hospitality. I thank Samuel Grushevsky and Mike Fried for
constructive remarks on this note.

\section {Thomae formula for cyclic covers and relations
between theta constants} We explain the general Thomae formula
following \cite{Na} for an algebraic curve $R$ given by the
equation:
$$
y^3 = \prod_{i=1}^{i=3m}(z-\lambda_i)(*)
$$
We denote $f:R\mapsto \mathbb{CP}^{1}$  the projection $(z,y)\mapsto
z.$ Define $Q_i=f^{-1}(\lambda_i),$ to be the unique branch point on
$R$ that is the pre image of $\lambda_i.$  Fix a homology basis
$a_1,a_2...a_{3m-2}, b_1,b_2,...,b_{3m-2}$ on $R$ such that the
intersections are $a_i a_j = 0 = b_i b_j $ and $a_i b_j = 1.$ Let
$v_1...v_{3m-2}$ be a basis of  standard holomorphic differentials
dual to the basis $a_1,a_2...a_{3m-2}, b_1,b_2,...,b_{3m-2}$ i.e.
$\int_{a_i}v_j=0, \int_{b_i}v_j=\delta_{ij}.$ Now fix an ordering of
$\lambda_i.$ Let $\phi$ be the automorphism of order $3$ defined by
$(z,y)\mapsto (z, \omega y)$ for $\omega^3=1.$ We write $\alpha
\equiv \beta$ for linear equivalent of divisors, i.e. if there
exists a function $g:R\mapsto CP^{1}$ and $div(g) = \alpha-\beta.$
The group $Div^0/\equiv$ is $Jac(R),$ the Jacobian of $R.$ ($Div^0$
- divisors of degree $0.$) Let $\psi$ be the mapping $\psi : Div
\mapsto Div/\equiv.$ Then the following lemma is true:

\begin{lem}
Let $P_1,P_2\in R, P_1\neq P_2$ and
$$D_i=P_i+\phi(P_i)+\phi^2(P_i), i=1,2$$ then $\psi(D_1)\equiv \psi(D_2).$
\end{lem}

\begin{proof}
Let $f_1(P) =\frac{f(P)-f(P_1)}{f(P)-f(P_2)}$ , then $div(f_1) =
D_1-D_2.$
\end{proof}

Define $D=\psi\left(P+\phi(P_i)+\phi^2(P_i)\right)$ as the
equivalence class in the Jacobian.

\begin{lem}
Let $K$ be the canonical divisor of $R$ Then the following holds:
\begin{enumerate}
    \item $$D\equiv 3Q_i \equiv  \infty_1+\infty_2+\infty_3$$
    \item $$K \equiv (2m-2)D$$
    \item $$\sum_1^{3m} Q_i \equiv mD$$
    \end{enumerate}
\end{lem}

\begin{proof} The first item follows exactly as in the previous
lemma. To show the rest, note that  $z\frac{dz}{w^2}$ is a
holomorphic differential with the divisor $Q_1^{6m-6}.$
\end{proof}

Now let $\Lambda=\left\{\Lambda_1, \Lambda_2, \Lambda_3\right\} $ be
a partition of $\left\{1,2,3,4,5,...3m\right\}$ with $|\Lambda_i|=m$
for $i=1,2,3$. We are interested in the following divisor
$e_\Lambda$ associated with the partition:
$$
  e_\Lambda= X_{\Lambda_1}+2X_{\Lambda_2} -D-\Delta
$$
where for each subset $S$ of $\left\{1,2...3m\right\}$ we set
$$
  X_S =  \sum_{Q_j\in S}Q_j
$$
Fix a point $P_0\in R$ and let $\Phi_{P_0}:R\to Jac(R)$ be given by
$\Phi_{P_0}\left(P\right)= \left(\int_{P_0}^{P} v_1...\int_{P_0}^{P}
v_{3m-2}\right)$.

\begin{defn}
Let $\mathbb{H}_g$ denote the set of $g\times g$ symmetric matrices,
$\tau$ such that the imaginary part of $\tau$ is positive definite.
For $\varepsilon , \varepsilon',\in\mathbb{R}^g$ and $\tau \in
\mathbb{H}_g$ we denote
$$
\theta\left[
\begin{array}{c}
  \varepsilon \\
  \varepsilon' \\
\end{array}
\right](\tau) = \sum_{l\varepsilon \mathbb{Z}^{2g}} \exp 2\pi i
\left\{\frac{1}{2}\left(l+\frac{\varepsilon}{2}\right)^t\tau\left(l+\frac{\varepsilon}{2}\right)+
\left(l+\frac{\varepsilon}{2}\right)^t\frac{\varepsilon'}{2}\right\}
$$
\end{defn}
This series is uniformly and absolutely convergent on compact
subsets of $\mathbb{C}^{g} \times \mathbb{H}_g.$ To each $w \in
\mathbb{C}^{3m-2}$ associate a unique  $w_1,w_2\in\mathbb{R}^g$ such
that $w=w_1+\tau w_2.$

\cite{Na} proves the following formula for theta constants with
characteristics associated to divisors $e_\Lambda.$ see \cite{BR} as
well:

\begin{thm}
The divisor $e_\Lambda$ is a point of order 6 on the Jacobian and
\begin{equation}\label{1}
    \theta[e_\Lambda]^6\left(\tau_R\right) = C_\Lambda (detA)^3 {\left((\Lambda_0\Lambda_0)(\Lambda_1\Lambda_1)(\Lambda_2\Lambda_2)\right)}^3(\Lambda_0\Lambda_1)(\Lambda_1\Lambda_2)(\Lambda_0\Lambda_2)
\end{equation}
Here $A$ is the matrix of certain differentials integrated with
respect to $a_i.$ and if
$$
  \Lambda_i =\left\{i_1<...<i_m\right\}, \Lambda_j=\left\{j_1<...<j_m\right\}
$$
then
$$
 (\Lambda_i\Lambda_i)=\prod_{k<l}\left(\lambda_{i_k}-\lambda_{i_l}\right) ,\quad
 (\Lambda_i\Lambda_j) =
 \prod_{k=1 , l=1}^m\left(\lambda_{i_k}-\lambda_{j_l}\right)
$$
\end{thm}

We apply the theorem to generate special relations between theta
functions with characteristics $e_\Lambda,$ evaluated at $\tau_R.$
For each partition $\Lambda$ denote the polynomial on the right hand
side of the last equation by $p_\Lambda.$  To obtain identities for
$\theta[e_\Lambda]$ we search for identities between $p_\Lambda.$
The key observation that allows us to simplify the problem is the
following form of the polynomials: choose $\Lambda
=\left\{\{1,2...,m\},\{m+1,...,2m\},\{2m,...,3m\}\right\}$. Then by
definition of $p_\Lambda$ the factor
$\prod\limits_{i=1}^3\prod\limits_{j=1}^3 \Lambda_i\Lambda_j$ is the
discriminant and a common factor for each $p_\Lambda$ which does not
depend on the partition $\Lambda$. Thus identities between
${\theta^6[e_\Lambda]}$ are equivalent to identities between the
polynomials
$$\left((\Lambda_0\Lambda_0)(\Lambda_1\Lambda_1)(\Lambda_2\Lambda_2)\right)^2.$$
Consequently, identities between $\sqrt{\theta^6[e_\Lambda]}$ are
equivalent to identities between the polynomials:
$$\left((\Lambda_0\Lambda_0)(\Lambda_1\Lambda_1)(\Lambda_2\Lambda_2)\right).$$
To get a hint for the result observe that the group $S_{3m}$ acts
naturally on the polynomials
$\left((\Lambda_0\Lambda_0)(\Lambda_1\Lambda_1)(\Lambda_2\Lambda_2)\right)$
via its action on the partitions of $\left\{1...3m\right\}.$ Thus
$Span(\left((\Lambda_0\Lambda_0)(\Lambda_1\Lambda_1)(\Lambda_2\Lambda_2)\right),$
is a vector space and has a representation of $S_{3m}$ on it.

\section{Explicit Basis}
In this section we provide an explicit basis for the space of
polynomials from the previous section. We imitate the process
described in [J] to construct a basis for the irreducible
representation of the symmetric group of $S_n$. For complex numbers
these representations are completely classified. We describe the
construction for any representation of the symmetric group and
obtain the relevant case of cyclic covers as an immediate corollary
of the general case. We remind the reader some facts from the
representation theory of $S_n$.

Let $n$ be a natural number and let $k_1...k_m$ be a partition of
$n.$ i.e. $\sum_{i=1}^m k_i=n$ and $k_1\geq k_2\geq k_3...\geq k_m.$
\begin{defn}
A Young diagram associated to a partition consists of $m$ rows such
that  $i$'th row has $k_i$ elements.
\end{defn}

\begin{defn}
Let $Y$ be a Young diagram; a tableau is obtained by distributing the
numbers $\left\{1...n\right\}$ within the  $m$ rows with the
following properties
\begin{itemize}
     \item Each row contains exactly $k_i$ elements
     \item The numbers in each row form an increasing sequence
\end{itemize}
\end{defn}
Assume that $\Lambda=\left\{\Lambda_0,...,\Lambda_k\right\}$ is a
tableau of $n.$  Define the polynomial:
$$
  (\Lambda_i\Lambda_i) = \prod_{i_k<i_l,\left\{i_k,i_l\right\} \in
  \Lambda_i }(\lambda_{i_k}-\lambda_{i_l}),\quad \hbox{where}\quad
  p_\Lambda=\prod_{i=1}^k(\Lambda_i\Lambda_i).
$$
The symmetric group, $S_n$ acts on $\Lambda$ and therefore acts on
the polynomials $p_\Lambda.$ To find the basis for $p_\Lambda$ we
use a modification of Garnier relation \cite{J} (7.1) to construct a
basis for the polynomials.\footnote{We were not able to find a
reference to our approach of constructing Specht modules though we
are confident its a folklore.} Arrange the tableau in columns ( i.e.
the first column will be elements of $\Lambda_1$ the second column
elements of $\Lambda_2$ etc). Overall we have $k$ columns for
$\Lambda.$ Let $X$ be a subset of the $i-th$ column of $\Lambda$ and
$Y$ is a subset of the $i+1-th$ column of $\Lambda.$ Let
$\sigma_1...\sigma_k$ be coset representatives for $S_{X \times Y}$
in $S_{X\bigcup Y}$. Then we have the Garnier relations:

\begin{thm}
Let $\mu_i$ denote the number of elements in the $i-th$ column of
$\Lambda.$ if $|X\bigcup Y| > \mu_i$ then
$$
  \sum_{m=1}^{k} sign (\sigma_m)\left(p_{\sigma_m\Lambda}\right)=0.
$$
\end{thm}
\begin{proof} If $|X\bigcup Y| > \mu_i$, by the pigeon
hole principle there exists an involution $\delta$ such that
$\sigma_m \Lambda$ is invariant under it. Thus
$$
 \sum_{m=1}^{k}
sign(\sigma_m)\left(p_{\sigma_m\Lambda}\right)= \sum_{m=1}^{k}
sign(\sigma_m)\left(p_{\delta\sigma_m\Lambda}\right) =$$
$$
    = - \sum_{m=1}^{k}sign\sigma_m\left(p_{\sigma_m\Lambda}\right)=0
$$
\end{proof}

In order to exhibit an explicit basis we define a standard Young
tableau
\begin{defn}
A standard tableau is a tableau where the rows and the columns are
arranged in an increasing order.
\end{defn}

\begin{defn}
We define an ordering on the set of tableaux by setting
$\Lambda^1<\Lambda^2$ if there is an $i$ such that
\begin{itemize}
  \item if $j > i$ than $j$ is in the same column of $\Lambda^1, \Lambda^2$
  \item $i$ is in more left column in $\Lambda^1$ than $\Lambda^2.$
\end{itemize}
\end{defn}

\begin{thm}
Let $\Lambda^1...\Lambda^k$ be the collection of standard tableaux
for a given partition. Then $p_{\Lambda^1}...p_{\Lambda^k}$ is a
basis for the vector space spanned by $\Lambda.$
\end{thm}

\begin{proof} We follow \cite{J} in the proof. We show that
$p_{\Lambda^k}$ spans any other polynomial corresponding to our
partition. Let $t$ be a tableau and suppose by induction that the
theorem is proved for each $t_1$ tableau such that $t_1<t.$ If $t$
is non standard there exists adjacent columns $a_1<...<a_q<...<a_r$
and $b_1<b_2...<b_q<...b_s$ such that $a_q>b_q$. Apply Garnier
relation for $X={a_1...a_r}, Y={b_1...b_q}.$ For each $\sigma$ a
representative in $S_{X \bigcup Y}$ in $S_{X\times Y}$ we have that
$[t\sigma] < t$ by the definition of the order $<.$ The result
follows immediately from the induction hypothesis.
\end{proof}

\begin{defn}
For an element $k$ of the  tableau $t$ Let $C_k,R_k$ be the unique
column and row $k$ belongs to. The hook of $k$, $h_k$ is the number
of elements beneath $k$ in $C_k$ plus the number of elements to the
right of $k$ in $R_k$ (include the element itself in the row but not
in the column.)
\end{defn}

It is well known that the number of standard tableaux equals to
\begin{equation}
\frac{n!}{\prod_k h_k}
\end{equation}
See \cite{J}.
\section{The ideal of theta identities}
We apply the theory of the previous paragraph to cyclic covers of
order 3. According to the theory, the hooks of the partitions
correspond to tableau with 3 rows and $m$ elements in each row. Our
first corollary is

\begin{cor}
The dimension of the polynomials $p_\Lambda$ (and hence the vector
space spanned by  $\sqrt{\theta^6[e_\Lambda]\left(\tau_R\right)}$
corresponding to them) is: $\frac{(3m)!\times 2}{(m+2)!(m+1)!m!}.$
\end{cor}

Hence we can also give a basis for
$\theta^6[e_\Lambda]\left(\tau_R\right)$ that correspond to the
different partitions $e_\Lambda.$
\begin{cor}
The set of $\sqrt{\theta^6[e_{\Lambda_S}]\left(\tau_R\right)},$
$\Lambda_S$ is a standard partition is a basis for a vector space
spanned by $\sqrt{\theta^6[e_\Lambda]\left(\tau_R\right)}.$ In
particular each $\sqrt{\theta^6[e_{\Lambda}]\left(\tau_R\right)},$
$\Lambda$ can be written as a linear combination of elements from
the set $\sqrt{\theta^6[e_{\Lambda_S}]\left(\tau_R\right)}.$
\end{cor}

\section{Example}

Let us revisit the case when there are 6 branch points and the genus
of the surface is 4. In this case, by the formula for the dimension,
the number of basis functions, $\theta^3[e_\Lambda]$ is : $2 \times
\frac{6!}{4!3!2!} = 5.$ We enumerate the 15 partitions as well as the
the polynomials that correspond to them:
\begin{enumerate}
\item $\Lambda=\left\{(1,2),(3,4),(5,6)\right\} p_\Lambda =
    (\lambda_1-\lambda_2)(\lambda_3-\lambda_4)(\lambda_5-\lambda_6)$
    \item $\Lambda=\left\{(1,2),(3,5),(4,6)\right\} p_\Lambda =
   (\lambda_1-\lambda_2)(\lambda_3-\lambda_5)(\lambda_4-\lambda_6)$
    \item $\Lambda=\left\{(1,2),(3,6),(4,5)\right\} p_\Lambda =
    (\lambda_1-\lambda_2)(\lambda_3-\lambda_6)(\lambda_4-\lambda_5)$
    \item $\Lambda=\left\{(1,3),(2,4),(5,6)\right\} p_\Lambda =
    (\lambda_1-\lambda_3)(\lambda_2-\lambda_4)(\lambda_5-\lambda_6)$
    \item  $\Lambda=\left\{(1,3),(2,5),(4,6)\right\} p_\Lambda =
    (\lambda_1-\lambda_3)(\lambda_2-\lambda_5)(\lambda_4-\lambda_6)$
    \item $\Lambda=\left\{(1,3),(2,6),(4,5)\right\} p_\Lambda =
            (\lambda_1-\lambda_3)(\lambda_2-\lambda_6)(\lambda_4-\lambda_5)$
    \item $\Lambda=\left\{(1,4),(2,5),(3,6)\right\} p_\Lambda =
    (\lambda_1-\lambda_4)(\lambda_2-\lambda_5)(\lambda_3-\lambda_6)$
    \item $\Lambda=\left\{(1,4),(2,6),(3,5)\right\} p_\Lambda =
    (\lambda_1-\lambda_4)(\lambda_2-\lambda_6)(\lambda_3-\lambda_5)$
    \item $\Lambda=\left\{(1,4),(2,3),(5,6)\right\} p_\Lambda =
    (\lambda_1-\lambda_4)(\lambda_2-\lambda_3)(\lambda_5-\lambda_6)$
    \item $\Lambda=\left\{(1,5),(2,3),(4,6)\right\} p_\Lambda =
    (\lambda_1-\lambda_5)(\lambda_2-\lambda_3)(\lambda_4-\lambda_6)$
    \item $\Lambda=\left\{(1,5),(2,4),(3,6)\right\} p_\Lambda =
    (\lambda_1-\lambda_5)(\lambda_2-\lambda_4)(\lambda_3-\lambda_6)$
    \item $\Lambda=\left\{(1,5),(2,6),(3,4)\right\} p_\Lambda =
    (\lambda_1-\lambda_5)(\lambda_2-\lambda_6)(\lambda_3-\lambda_4)$
    \item $\Lambda=\left\{(1,6),(2,3),(4,5)\right\} p_\Lambda =
    (\lambda_1-\lambda_6)(\lambda_2-\lambda_3)(\lambda_4-\lambda_5)$
    \item $\Lambda=\left\{(1,6),(2,4),(3,5)\right\} p_\Lambda =
    (\lambda_1-\lambda_6)(\lambda_2-\lambda_4)(\lambda_3-\lambda_5)$
    \item $\Lambda=\left\{(1,6),(2,5),(3,4)\right\} p_\Lambda =
    (\lambda_1-\lambda_6)(\lambda_2-\lambda_5)(\lambda_3-\lambda_4)$
\end{enumerate}
The basis for the vector space of the polynomials corresponds to the
following standard tableaux:

\begin{enumerate}
    \item $\Lambda=\left\{(1,2),(3,4),(5,6)\right\}p_\Lambda =
    (\lambda_1-\lambda_2)(\lambda_3-\lambda_4)(\lambda_5-\lambda_6)$
    \item $\Lambda=\left\{(1,2),(3,5),(4,6)\right\} p_\Lambda =
    (\lambda_1-\lambda_2)(\lambda_3-\lambda_5)(\lambda_4-\lambda_6)$
    \item $\Lambda=\left\{(1,3),(2,4),(5,6)\right\} p_\Lambda =
    (\lambda_1-\lambda_3)(\lambda_2-\lambda_4)(\lambda_5-\lambda_6)$
    \item  $\Lambda=\left\{(1,3),(2,5),(4,6)\right\} p_\Lambda =
    (\lambda_1-\lambda_3)(\lambda_2-\lambda_5)(\lambda_4-\lambda_6)$
    \item $\Lambda=\left\{(1,4),(2,5),(3,6)\right\}p_\Lambda =
    (\lambda_1-\lambda_4)(\lambda_2-\lambda_5)(\lambda_3-\lambda_6)$
\end{enumerate}
The rest of the $10$ polynomials can be rewritten as a linear
combination of the set above applying Garnier's algorithm as in
Theorem 3.7. For example we have:
$$
    (\lambda_1-\lambda_2)(\lambda_3-\lambda_6)(\lambda_4-\lambda_5)=
    -(\lambda_1-\lambda_2)(\lambda_3-\lambda_4)(\lambda_5-\lambda_6)+
    (\lambda_1-\lambda_2)(\lambda_3-\lambda_5)(\lambda_4-\lambda_6)
$$
$$ (\lambda_1-\lambda_3)(\lambda_2-\lambda_6)(\lambda_4-\lambda_5)=
    -(\lambda_1-\lambda_3)(\lambda_2-\lambda_4)(\lambda_5-\lambda_6)+
    (\lambda_1-\lambda_3)(\lambda_2-\lambda_5)(\lambda_4-\lambda_6)
$$
$$(\lambda_1-\lambda_6)(\lambda_2-\lambda_5)(\lambda_3-\lambda_4) =
(\lambda_1-\lambda_4)(\lambda_2-\lambda_5)(\lambda_3-\lambda_6)
-(\lambda_1-\lambda_3)(\lambda_2-\lambda_5)(\lambda_4-\lambda_6)
$$

The others polynomials can be expressed in a similar way leading to
identities between $\sqrt{\theta^6[e_\Lambda]\left(\tau_R\right)}$
in this case. Let us conclude with the following remarks on the
identities above: In the hyperelliptic curve case the identities
between integral characteristics of theta functions evaluated at
period matrix of hyperelliptic curves arise from vanishing
properties of theta functions. In our case it is interesting to
investigate whether an analogous situation can arise. The only
source of cubic theta identities known to the author, is the
following theorem in \cite{Ko}:
\begin{thm}
Let $ \left [ \begin{array}{cc}\mu\\\mu'\end{array}\right]$ be an
odd integral theta characteristics in genus $3m-2$ Then for any
$\tau \in \mathbb{H}_{3m-2}$:
\begin{equation}
    \sum_{0\leq\nu_i\leq 3}(-1)^{\sum_{i=1}^{3m-2}\mu_i\nu_i}\theta^3\left [ \begin{array}{cc}\mu\\
\mu'+\frac{2\nu}{3}\end{array}\right](0,\tau) = 0
\end{equation}

\end{thm}

It is plausible that the vanishing of theta constants with rational
characteristics of order 3 on $\tau_R$ will produce a new proof for
the special identities obtained in this note using Thomae formula.
Finally note that for all the identities (4) the coefficients are
$\pm 1$ It is plausible that this a general phenomenon.

\section{conclusion}
There exists an extensive literature on Schottky-Jung identities and
on theta constants for hyperelliptic curves. In this note we
obtained special identities for other classes of algebraic curves.
In subsequent notes we plan to pursue and develop further the themes
touched in this note, especially applications of similar methods to
general Hurwitz spaces and their  mapping class groups.

\end{document}